\documentclass[12pt]{amsart}
\usepackage{amsmath, amsfonts, amssymb, amstext, amscd, amsthm, supertabular, multicol}
\usepackage[all,2cell,ps]{xy}
\usepackage{verbatim} 

\newtheorem{3squares}{Proposition}
\newtheorem{congruum}[3squares]{Corollary}
\newtheorem{4squares}[3squares]{Theorem}
\newtheorem{conic}[3squares]{Proposition}
\newtheorem{conicsquares}[3squares]{Theorem}
\newtheorem{4torsion}[3squares]{Proposition}
\newtheorem{symmetry}[3squares]{Proposition}
\newtheorem{coniccongruum}[3squares]{Corollary}

\begin{document}

\title{Arithmetic Progressions on Conic Sections}

\author{Alejandra Alvarado}
\address{Department of Mathematics \\ Purdue University \\ 150 North University Street \\ West Lafayette, IN 47907}
\email{alvaraa@math.purdue.edu}

\author{Edray Herber Goins}
\address{Department of Mathematics \\ Purdue University \\ 150 North University Street \\ West Lafayette, IN 47907}
\email{egoins@math.purdue.edu}



\maketitle

\begin{abstract}
The set $\{1, 25, 49\}$ is a 3-term collection of integers which forms an arithmetic progression of perfect squares.  We view the set $\bigl \{(1,1), \, (5,25), \, (7,49) \bigr \}$ as a 3-term collection of rational points on the parabola $y=x^2$ whose $y$-coordinates form an arithmetic progression.  In this exposition, we provide a generalization to 3-term arithmetic progressions on arbitrary conic sections $\mathcal C$ with respect to a linear rational map $\ell: \mathcal C \to \mathbb P^1$.  We explain how this construction is related to rational points on the universal elliptic curve $Y^2 + 4 \, X \, Y + 4 \, k \, Y = X^3 + k \, X^2$ classifying those curves possessing a rational 4-torsion point. \end{abstract}

\section{Introduction}
An $n$-term arithmetic progression is a collection of rational numbers $\{\ell_1, \ell_2, ...,  \ell_n \}$ such that there is a common difference $\delta = \ell_{i+1}-\ell_i$.   The set $\{1, 25, 49\}$ is a 3-term collection of integers which forms an arithmetic progression of perfect squares.  We view the set $\bigl \{(1:1:1), \, (5:25:1), \, (7:49:1) \bigr \}$ as a 3-term collection of rational points $(x:y:1)$ on the parabola $y=x^2$ whose $y$-coordinates form an arithmetic progression.  It is well-known that there are infinitely many such progressions of points on the parabola:
\[ \{ P_1, \, P_2, \, P_3 \} = \left \{ \begin{matrix}
\left( t^2 - 2 \, t - 1 \, : \, \bigl( t^2 - 2 \, t - 1 \bigr)^2 \, : \, 1 \right), \\[10pt]
\left( t^2 + 1 \, : \, \bigl( t^2 + 1 \bigr)^2 \, : \, 1 \right), \\[10pt]
\left( t^2 + 2 \, t - 1 \, : \, \bigl( t^2 + 2 \, t - 1 \bigr)^2 \, : \, 1 \right)
\end{matrix} \right \}. \]

\noindent In this article, we consider the question of forming a 3-term arithmetic progression on an arbitrary conic section.

Here is our main result.  We will say that $\{ P_1, \, P_2, \, P_3 \} \subseteq \mathbb P^2(K)$ forms an arithmetic progression on a conic section
\[ \mathcal C = \left \{ (x_1: x_2 : x_0) \in \mathbb P^2 \ \left| \ \begin{aligned} A \, x_1^2 & + 2 \, B \, x_1 \, x_2 + C \, x_2^2 \\ & + 2 \, D \, x_1 \, x_0 + 2 \, E \, x_2 \, x_0 + F \, x_0^2 = 0 \end{aligned} \right. \right \} \]

\noindent with respect to a linear rational map 
\[ \begin{CD} \mathcal C(K) @>{\ell}>> \mathbb P^1(K), \qquad (x_1 : x_2 : x_0) \mapsto \dfrac {a \, x_1 + b \, x_2 + c \, x_0}{d \, x_1 + e \, x_2 + f \, x_0} \end{CD} \]

\noindent defined over a field $K$ of characteristic different from 2 if there is a common difference $\delta = \ell(P_2) - \ell(P_1) = \ell(P_3) - \ell(P_2)$.  Consider those $t_0 \in K$ such that $\sqrt{\text{Disc}(t_0)} \in K$ and
\[ k = \dfrac {\text{Disc}'(t_0)^2 - 2 \, \text{Disc}(t_0) \, \text{Disc}''(t_0)}{\text{Disc}'(t_0)^2} \neq 0, \, 1, \, \infty \]

\noindent in terms of the discriminant
\[ \begin{aligned} & \text{Disc}(t) \\ & \quad = \left[ \begin{matrix} a-d \, t \\ b-e \, t \\ c-f \, t \end{matrix} \right]^T \left[ \begin{matrix} E^2 - C \, F & B \, F - D \, E & C \, D - B \, E \\ B \, F - D \, E & D^2 - A \, F & A \, E - B \, D \\ C \, D - B \, E & A \, E - B \, D & B^2 - A \, C \end{matrix} \right] \left[ \begin{matrix} a-d \, t \\ b-e \, t \\ c-f \, t \end{matrix} \right]. \end{aligned} \]

\noindent For each $K$-rational point $(X:Y:1)$ on the elliptic curve
\[ \mathcal E_k: \quad Y^2 + 4 \, X \, Y + 4 \, k \, Y = X^3 + k \, X^2 \]

\noindent the desired set is
\[ \{ P_1, \, P_2, \, P_3 \} = \left \{ \begin{matrix} \bigl( x_1(t_0 - \delta):x_2(t_0-\delta):x_0(t_0-\delta) \bigr), \\[10pt] \bigl( x_1(t_0):x_2(t_0):x_0(t_0) \bigr), \\[10pt] \bigl( x_1(t_0 + \delta):x_2(t_0+\delta):x_0(t_0+\delta) \bigr) \end{matrix} \right \} \]

\noindent in terms of the coordinates

\[ \begin{aligned}
x_1(t) & = B \, (b-e \, t) \, (c-f \, t) - C\, (a-d \, t) \, (c-f \, t) \\ & - D \, (b-e \, t)^2 + E \, (a-d \, t) \, (b-e \, t) \pm (b - e \, t) \, \sqrt{\text{Disc}(t)}, \\[5pt]
x_2(t) & = -A \, (b - e \, t) \, (c - f \, t ) + B \, (a-d \, t ) \, (c-f \, t ) \\ & + D \, (a-d \, t ) \, (b-e \, t) - E \, (a-d \, t)^2 \mp (a-d \, t)  \, \sqrt{\text{Disc}(t)}, \\[5pt]
x_0(t) & = A \, (b-e \, t )^2 - 2 \, B \, (a-d \, t) \, (b-e \, t) + C \, (a-d \, t )^2;
\end{aligned} \]

\noindent and the common difference
\[ \delta = - \dfrac {\text{Disc}(t_0)}{\text{Disc}'(t_0)} \, \dfrac {4 \, X \, Y}{Y^2 + 2 \, X \, Y + k \, X^2}. \]

\noindent For example, the special case of squares in arithmetic progression follows from consideration of the parabola $\mathcal C: y = x^2$ and the linear polynomial $\ell(x, y, 1) = y$, so that $\text{Disc}(t) = t$.  We may choose $t_0 = 1$ to find the previously mentioned 3-term progression of squares.

The condition $\sqrt{\text{Disc}(t_0)} \in K$ is both necessary and sufficient for the existence of points $P_0 \in \mathcal C(K)$ such that $\ell(P_0) = t_0$, while the condition $k \neq 0, \, 1, \, \infty$ is both necessary and sufficient for $\mathcal E_k$ to be an elliptic curve.  Surprisingly, this curve is universal in that is classifies those elliptic curves possessing a rational 4-torsion point.

\section{Squares in Arithmetic Progression}

Let $K$ be a field of characteristic different from 2, which in practice is either $\mathbb Q$ or a quadratic extension $\mathbb Q(\sqrt{k})$ thereof.  We will say that a collection $\{ y_1, \, y_2, \, ..., \, y_n \} \subseteq K$ is an $n$-term arithmetic progression of squares over $K$ if 
\begin{enumerate}
\item[(i)] each $y_i=x_i^2$ for some $x_i \in K$, and 
\item[(ii)] there is a common difference $\delta = y_{i+1}-y_i$.
\end{enumerate}
\noindent For example, the set $\{1, 25, 49\}$ is a 3-term collection of integers which forms an arithmetic progression, where the common difference is $\delta = 24$.  Most of the results which are contained in this section are well-known in the literature.  See for instance Gonz{\'a}lez-Jim{\'e}nez and Steuding \cite{GonzalezJimenez:2010}, Gonz{\'a}lez-Jim{\'e}nez and Xarles \cite{GonzalezJimenez:2009vja}, and Xarles \cite{Xarles:2009tc}.  We include these results in the exposition at hand in order to both include simplified proofs and motivate generalizations.

\subsection{Three Squares in Arithmetic Progression}

Let us first classify the 3-term case.

\begin{3squares} \label{3squares} $\{ y_1, \, y_2, \, y_3\} \subseteq K$ are three squares in an arithmetic progression if and only if there exists $t \in K$ such that
\[ \bigl( y_1 \, : \, y_2 \, : \, y_3 \bigr) = \left( \bigl( t^2 - 2 \, t - 1 \bigr)^2 \, : \, \bigl( t^2 + 1 \bigr)^2 \, : \, \bigl( t^2 + 2 \, t - 1 \bigr)^2 \right). \] \end{3squares}

\begin{proof} Set $y_1 = x_1^2$, $y_2 = x_2^2$, and $y_3 = x_3^2$.  We will use the following equations:
\[ \left. \begin{aligned} x_1^2 - 2 \, x_2^2 + x_3^2 & = 0 \\[5pt] \dfrac {x_1 - x_3}{x_1 - 2 \, x_2 + x_3} & = t \end{aligned} \right \} \qquad \iff \qquad \left \{ \begin{aligned} \dfrac {x_1}{x_2} & = \dfrac {t^2 - 2 \, t - 1}{t^2 + 1} \\[5pt] \dfrac {x_3}{x_2} & = \dfrac {t^2 + 2 \, t - 1}{t^2 + 1} \end{aligned} \right. \]
\noindent If the ratio $(y_1 : y_2 : y_3)$ is as above, then $\delta = 4 \, (t^3 - t) / (t^2 + 1)^2 \cdot x_2^2$ is the common difference.  Conversely, $\{ y_1, \, y_2, \, y_3\} \subseteq K$ are three squares in an arithmetic progression, then the ratio $(y_1 : y_2 : y_3)$ is as above for $t = (x_1 - x_3) / (x_1 - 2 \, x_2 + x_3)$.  \end{proof}

\noindent There are only certain common differences $\delta$ which can occur for such an arithmetic progression.

\begin{congruum} \label{congruum} The following are equivalent for each nonzero $\delta \in K$.
\begin{enumerate} 
\item There exist $x_1, \, x_2, \, x_3 \in K$ such that $x_2^2 - \delta = x_1^2$ and $x_2^2 + \delta = x_3^2$.
\item There exist $X, \, Y \in K$ with $Y \neq 0$ such that $Y^2 = X^3 - \delta^2 \, X$.
\item There exist $a, \, b, \, c \in K$ such that $a^2 + b^2 = c^2$ and $(1/2) \, a\, b = \delta$.
\end{enumerate}
\end{congruum}

\begin{proof} We found above that there exist $x_1, \, x_2, \, x_3 \in K$ such that $\delta = x_2^2 - x_1^2 = x_3^2 - x_2^2$ if and only if $\delta = 4 \, (t^3 - t) / (t^2 + 1)^2 \cdot x_2^2$ for some $t \in K$.  This motivates the following transformation:
\[ \left. \begin{aligned} X & = \dfrac {(x_1 - x_3) \, \delta}{x_1 - 2 \, x_2 + x_3} \\[5pt] Y & = - \dfrac {2 \, \delta^2}{x_1 - 2 \, x_2 + x_3} \end{aligned} \right \} \qquad \iff \qquad \left \{ \begin{aligned} x_1 & = \dfrac {X^2 - 2 \, \delta \, X - \delta^2}{2 \, Y} \\[5pt] x_2 & = \dfrac {X^2 + \delta^2}{2 \, Y} \\[5pt] x_3 & = \dfrac {X^2 + 2 \, \delta \, X - \delta^2}{2 \, Y} \end{aligned} \right. \]
\noindent Hence $\delta = x_2^2 - x_1^2 = x_3^2 - x_2^2$ if and only if $Y^2 = X^3 - \delta^2 \, X$ for some nonzero $Y$. Similarly, we have the following transformation:
\[ \left. \begin{aligned} a & = \dfrac {X^2 - \delta^2}{Y} \\[5pt] b & = \dfrac {2 \, \delta \, X}{Y} \\[5pt] c & = \dfrac {X^2 + \delta^2}{Y} \end{aligned} \right \} \qquad \iff \qquad \left \{ \begin{aligned} X & = \dfrac {b \, \delta}{c-a} \\[5pt] Y & = \dfrac {2\, \delta^2}{c-a} \end{aligned} \right. \]

\noindent Hence $Y^2 = X^3 - \delta^2 \, X$ for some nonzero $Y$ if and only if $a^2 + b^2 = c^2$ and $a \, b = 2 \, \delta$. \end{proof}

\noindent Any nonzero $\delta \in K$ satisfying the equivalent criteria above is called a \emph{congruum}, or even a \emph{congruent number}, although this notation is perhaps more standard for $K = \mathbb Q$.  Chandrasekar \cite{Chandrasekar:1998jw} and Izadi \cite{Izadi:2010vs} have written a nice summary of the history of the problem of determining when $\delta \in \mathbb Z_{>0}$ is a congruent number.  This notation comes from Fibonacci's 1225 work \emph{Liber Quadratorum}, where he shows $\delta = 5, \, 6, \, 7$ are congruent numbers, although similar results appeared much earlier with Al-Karaji some 200 years before.  In 1640, Fermat showed that $\delta = 1, \, 2, \, 3, \, 4$ are not congruent numbers.  In 1972, Alter, Curtz, and Kubota \cite{Alter:1972ur} conjectured that if $\delta$ is an integer congruent to $5, \, 6, \, 7 \pmod{8}$ then $\delta$ is a congruent number.  In 1983, Tunnell \cite{Tunnell:1983ki} found a necessary and sufficient condition for $\delta$ to be a congruent number and translated this into properties of the quadratic twists of the elliptic curve $Y^2 = X^3 - X$.  For example, if $\delta \equiv 5, \, 6, \, 7 \pmod{8}$ then $Y^2 = X^3 - \delta^2 \, X$ should have positive rank.

\subsection{Four or More Squares in Arithmetic Progression}

Fermat's work is not limited to congruent numbers.  He also stated that there are no non-constant $n$-term arithmetic progressions whose terms are perfect squares over $\mathbb Q$ if $n\geq 4$.  Euler gave a rigorous proof of this claim in 1780.  We recast this in the modern language of elliptic curves.

\begin{4squares} \label{4squares} Let $K$ be a field of characteristic different from 2.
\begin{enumerate}
\item $\{ y_1, \, y_2, \, y_3, \, y_4 \} \subseteq K$ are four squares in an arithmetic progression if and only if the elliptic curve $Y^2 = X^3 + 5 \, X^2 + 4 \, X$ has a $K$-rational point $(X:Y:1)$.
\item If point $(X : Y : 1)$ is a $K$-rational point on the quadratic twist $Y^2 = X^3 + 5 \, k \, X^2 + 4 \, k^2 \, X$, then the following are four squares in an arithmetic progression over $K(\sqrt{k})$:
\[ \begin{aligned} 
y_1 & = \left( 3 \, k \, X \, (2 \, k + X) - 2 \, \sqrt{k} \, Y \, (2 \, k - X) \right)^2 \\
y_2 & = \left(      k \, X \, (2 \, k - X) - 2 \, \sqrt{k} \, Y \, (2 \, k + X) \right)^2 \\ 
y_3 & = \left(      k \, X \, (2 \, k - X) + 2 \, \sqrt{k} \, Y \, (2 \, k + X) \right)^2 \\
y_4 & = \left( 3 \, k \, X \, (2 \, k + X) + 2 \, \sqrt{k} \, Y \, (2 \, k - X) \right)^2
\end{aligned} \]
\item If $\{ y_1, \, y_2, \,\dots, \, y_n \} \subseteq \mathbb Q$ are  $n \geq 4$ squares in an arithmetic progression over $\mathbb Q$, then $(y_1 : y_2 : \cdots : y_n) = (1:1:\cdots:1)$.
\end{enumerate}
\end{4squares}

\begin{proof} Set $y_1 = x_1^2$, $y_2 = x_2^2$, $y_3 = x_3^2$, and $y_4 = x_4^2$.  For the first statement, use the transformation
\[ \left. \begin{aligned}
x_1^2 - 3 \, x_2^2 + 3 \, x_3^2 - x_4^2 & = 0 \\[5pt]
2 \, \dfrac {x_1 - 3 \, x_2 - 3 \, x_3 + x_4}{x_1 + 3 \, x_2 + 3 \, x_3 + x_4} & = X \\[5pt]
6 \, \dfrac {x_1 - x_2 + x_3 - x_4}{x_1 + 3 \, x_2 + 3 \, x_3 + x_4} & = Y
\end{aligned} \right \} \iff \left \{ \begin{aligned}
\dfrac {x_1}{x_4} & = \dfrac {6 \, X + 3 \, X^2 - 2 \, Y + X \, Y}{6 \, X + 3 \, X^2 + 2 \, Y - X \, Y} \\[5pt]
\dfrac {x_2}{x_4} & = \dfrac {2 \, X -      X^2 - 2 \, Y - X \, Y}{6 \, X + 3 \, X^2 + 2 \, Y - X \, Y} \\[5pt]
\dfrac {x_3}{x_4} & = \dfrac {2 \, X -      X^2 + 2 \, Y + X \, Y}{6 \, X + 3 \, X^2 + 2 \, Y - X \, Y}
\end{aligned} \right. \]
\noindent The equations $x_1^2 - 2 \, x_2^2 + x_3^2 = x_2^2 - 2 \, x_3^2 + x_4^2 = 0$ are equivalent to the one equation $Y^2 = X^3 + 5 \, X^2 + 4 \, X$.  For the second statement, use the fact that $\bigl( k \, U :  \sqrt{k} \, V : k^2 \bigr)$ is a $K(\sqrt{k})$-rational point on the elliptic curve $Y^2 = X^3 + 5 \, X^2 + 4 \, X$ whenever $(U:V:1)$ is a $K$-rational point on the quadratic twist $V^2 = U^3 + 5 \, k \, U^2 + 4 \, k^2 \, U$.  For the last statement, we use the rational points in Table \ref{points}: the right-hand column is the exhaustive list of all $\mathbb Q$-rational points $(X:Y:1)$ on $Y^2 = X^3 + 5 \, X^2 + 4 \, X$.  \end{proof}

\begin{table}[h] \caption{Rational Points on $Y^2 = X^3 + 5 \, X^2 + 4 \, X$} \label{points} 
\begin{center} \begin{tabular}{|c|c|}
\hline $(x_1 : x_2 : x_3 : x_4)$ & $(X : Y : 1)$ \\ \hline
$(-1 : -1 : +1 :+1)$ & $(0:1:0)$ \\
$(-1 : +1 : -1 :+1)$ & $(0:0:1)$ \\
$(-1 : -1 : -1 :+1)$ & $(-2:+2:1)$ \\
$(-1 : +1 : +1 :+1)$ & $(-2:-2:1)$ \\
$(+1 : +1 : +1 :+1)$ & $(-1:0:1)$ \\
$(+1 : -1 : -1 :+1)$ & $(-4:0:1)$ \\
$(+1 : +1 : -1 :+1)$ & $( 2:+6:1)$ \\
$(+1 : -1 : +1 :+1)$ & $( 2:-6:1)$ \\ \hline
\end{tabular} \end{center} 
\end{table}

The elliptic curve $Y^2 = X^3 + 5 \, X^2 + 4 \, X$ is the same as the modular curve $X_0(24)$, where $X_0 \bigl( 24 \bigr)(\mathbb Q) \simeq Z_2 \times Z_4$ is a finite group.  (We will provide a formal definition of modular curves in the next section.)  The results above give a sufficient condition for there to be nonconstant four-term arithmetic progressions whose terms are perfect squares $K = \mathbb Q(\sqrt{k})$ by considering $K$-rational points on $X_0(24)^{(k)}$.  For example, when $k = 6$, the quadratic twist $Y^2 = X^3 + 5 \, k \, X^2 + 4 \, k^2 \, X$ has a $\mathbb Q$-rational point $(X:Y:1) = (-8:-16:1)$, so that we find an arithmetic progression of four squares over $K = \mathbb Q(\sqrt{6})$:
\[ \begin{aligned} & \bigl( y_1 : y_2 : y_3 : y_4 \bigr) \\ & \qquad = \bigl( (9-5 \, \sqrt{6})^2 \ : \ (15 - \sqrt{6})^2 \ : \ (15 + \sqrt{6})^2 \ : \ (9 + 5 \, \sqrt{6})^2 \bigr). \end{aligned} \]
\noindent However, not all arithmetic progressions of four squares over $K = \mathbb Q(\sqrt{k})$ are in this form.  In 2009, Gonz{\'a}lez-Jim{\'e}nez and Xarles \cite{GonzalezJimenez:2009vja} found the following arithmetic progression of five squares over $K = \mathbb Q(\sqrt{409})$:
\[ (y_1 : y_2 : y_3 : y_4: y_5) = \bigl( 7^2 \ : \ 13^2 \ : \ 17^2 \ : \ (\sqrt{409})^2\ : \ 23^2 \bigr). \]

\section{Arithmetic Progressions on Conic Sections}

The set $\{1, 25, 49\}$ is a 3-term collection of integers which forms an arithmetic progression.  We view the set $\bigl \{(1:1:1), \, (5:25:1), \, (7:49:1) \bigr \}$ as a 3-term collection of rational points $(x:y:1)$ on the parabola $y=x^2$ whose $y$-coordinates form an arithmetic progression.  This motivates a natural generalization of arithmetic progressions on conics.  Consider an arbitrary conic section as viewed in the projective plane:
\[ \mathcal C = \left \{ (x_1: x_2 : x_0) \in \mathbb P^2 \ \left| \ \begin{aligned} A \, x_1^2 & + 2 \, B \, x_1 \, x_2 + C \, x_2^2 \\ & + 2 \, D \, x_1 \, x_0 + 2 \, E \, x_2 \, x_0 + F \, x_0^2 = 0 \end{aligned} \right. \right \}. \]

\noindent Upon fixing a linear rational map $\ell$, we have a map
\[ \begin{CD} \mathcal C(K) @>{\ell}>> \mathbb P^1(K), \qquad (x_1 : x_2 : x_0) \mapsto \dfrac {a \, x_1 + b \, x_2 + c \, x_0}{d \, x_1 + e \, x_2 + f \, x_0}. \end{CD} \]
We will call a set $\bigl \{ P_1, \, P_2, \, \dots, \, P_n \bigr \} \subseteq \mathbb P^2(K)$ an \emph{arithmetic progression on $\mathcal C$ with respect to $\ell$} if 
\begin{itemize}
\item[(i)] each $P_i \in \mathcal C(K)$ is a point on the conic section, and 
\item[(ii)] there is a common difference $\delta = \ell(P_{i+1}) - \ell(P_i)$.
\end{itemize}
For \emph{squares} in an arithmetic progression we would choose the parabola $\mathcal C: y = x^2$ and the linear polynomial $\ell(x,y,1) = y$.  We are interested in the intersection of a line $\ell(x_1, x_2, x_0) = t$ with the conic $\mathcal C(K)$ where we allow $t \in K$ to vary. 

\begin{conic} \label{conic} Fix a conic section
\[ \mathcal C: \quad A \, x^2 + 2 \, B \, x \, y + C \, y^2 + 2 \, D \, x + 2 \, E \, y + F = 0 \]
\noindent and a linear rational map $\ell(x,y,1) = \bigl( a \, x + b \, y + c \bigr)/\bigl(d \, x + e \, y + f\bigr)$ over a field $K$ of characteristic different from 2.  For any $t \in K$, define the discriminant via the quadratic polynomial
\[ \begin{aligned} & \text{Disc}(t) \\ & \quad = \left[ \begin{matrix} a-d \, t \\ b-e \, t \\ c-f \, t \end{matrix} \right]^T \left[ \begin{matrix} E^2 - C \, F & B \, F - D \, E & C \, D - B \, E \\ B \, F - D \, E & D^2 - A \, F & A \, E - B \, D \\ C \, D - B \, E & A \, E - B \, D & B^2 - A \, C \end{matrix} \right] \left[ \begin{matrix} a-d \, t \\ b-e \, t \\ c-f \, t \end{matrix} \right]. \end{aligned} \]
\noindent If there is a point $P = (x_1 : x_2 : x_0)$ in $\mathcal C(K)$ satisfying $\ell(P) = t$, then $\sqrt{\text{Disc}(t)} \in K$, and the coordinates must be of the form
\[ \begin{aligned}
x_1(t) & = B \, (b-e \, t) \, (c-f \, t) - C\, (a-d \, t) \, (c-f \, t) \\ & - D \, (b-e \, t)^2 + E \, (a-d \, t) \, (b-e \, t) \pm (b - e \, t) \, \sqrt{\text{Disc}(t)}, \\[5pt]
x_2(t) & = -A \, (b - e \, t) \, (c - f \, t ) + B \, (a-d \, t ) \, (c-f \, t ) \\ & + D \, (a-d \, t ) \, (b-e \, t) - E \, (a-d \, t)^2 \mp (a-d \, t)  \, \sqrt{\text{Disc}(t)}, \\[5pt]
x_0(t) & = A \, (b-e \, t )^2 - 2 \, B \, (a-d \, t) \, (b-e \, t) + C \, (a-d \, t )^2.
\end{aligned} \]
\noindent Conversely, if $\sqrt{\text{Disc}(t)} \in K$, then $P = \bigl( x_1(t) : x_2(t) : x_0(t) \bigr)$ is a point in $\mathcal C(K)$ satisfying $\ell(P) = t$.
\end{conic}

\begin{proof} It is easy to verify that $x_1 = x_1(t)$, $x_2 = x_2(t)$, and $x_0 = x_0(t)$ as above is to the solution to the simultaneous equations
\[ \begin{aligned} \bigl( a - d \, t \bigr) \, x_1 + \bigl( b - e \, t \bigr) \, x_2 + \bigl( c - f \, t \bigr) \, x_0 & = 0 \\ A \, x_1^2 + 2 \, B \, x_1 \, x_2 + C \, x_2^2 + 2 \, D \, x_1 \, x_0 + 2 \, E \, x_2 \, x_0 + F \, x_0^2 & = 0 \end{aligned} \]

\noindent Similarly, it is easy to verify that the $3 \times 3$ determinant
\[ \left| \begin{matrix} a & b & c \\[5pt] d & e & f \\[5pt] \dfrac {A \, x_1 + B \, x_2 + D \, x_0}{d \, x_1 + e \, x_2 + f \, x_0} & \dfrac {B \, x_1 + C \, x_2 + E \, x_0}{d \, x_1 + e \, x_2 + f \, x_0} & \dfrac {D \, x_1 + E \, x_2 + F \, x_0}{d \, x_1 + e \, x_2 + f \, x_0} \end{matrix} \right| \]
\noindent has the value $\pm \sqrt{\text{Disc}(t)}$, and hence must be an element of $K$. \end{proof}

We explain how Proposition \ref{conic} is related to Proposition \ref{3squares} and Theorem \ref{4squares}.  For the parabola $\mathcal C: \, y = x^2$ and the linear polynomial $\ell(x,y,1) = y$, we have the discriminant $\text{Disc}(t) = t$.  The corresponding points on the parabola are $P = \bigl( \pm \sqrt{t} : t : 1 \bigr)$.  In order to have an arithmetic progression $\{ y_1, \, y_2, \,\dots, \, y_n \} \subseteq K$ of squares, we must choose $y_i = t_i = x_i^2$.  The difficulty in creating an arithmetic progression is forcing $\delta = t_{i+1} - t_i$ to be a constant.

\subsection{Three Squares in Arithmetic Progression Revisited}

The following result gives a generalization to the infinitude of three rational squares in an arithmetic progression.  

\begin{conicsquares} \label{conicsquares} Continue notation as in Proposition \ref{conic}.  Assume that there exists $t_0 \in K$ such that $\sqrt{\text{Disc}(t_0)} \in K$, and define the quantity
\[ k = \dfrac {\text{Disc}'(t_0)^2 - 2 \, \text{Disc}(t_0) \, \text{Disc}''(t_0)}{\text{Disc}'(t_0)^2} 
 \in K. \]
\begin{enumerate}
\item There exist nontrivial sequences $\{ t_1, \, t_2, \, \dots, \, t_n \} \subseteq K$ such that each $\sqrt{\text{Disc}(t_i)} \in K$.  In particular, each $K$-rational point $P_i = \bigl( x_1(t_i) : x_2(t_i) : x_0(t_i) \bigr)$ is in $\mathcal C(K)$, and $\ell(P_{i+1}) - \ell(P_i) = t_{i+1} - t_i$. \\
\item There exists an arithmetic progression $\{ P_1, \, P_2, \, P_3 \} \subseteq \mathbb P^2(K)$ on $\mathcal C$ with respect to $\ell$.  Moreover, the points must necessarily be of the form $P_i = \bigl( x_1(t_i) : x_2(t_i) : x_0(t_i) \bigr)$ in terms of
\[ \begin{aligned} t_1 & = t_0 - \delta \\ t_2 & = t_0 \\ t_3 & = t_0 + \delta \end{aligned} \quad \text{where} \quad \delta = - \dfrac {\text{Disc}(t_0)}{\text{Disc}'(t_0)} \, \dfrac {4 \, X \, Y}{Y^2 + 2 \, X \, Y + k \, X^2} \]
\noindent  for some $K$-rational point $(X:Y:1)$ on the cubic curve
\[ \mathcal E_k: \quad Y^2 + 4 \, X \, Y + 4 \, k \, Y = X^3 + k \, X^2. \]
\end{enumerate}
\end{conicsquares}

\begin{proof} Assume that $\sqrt{\text{Disc}(t)} \in K$ for some $t \in K$, and define the quantities
\[ 
c_n = \dfrac {1}{\text{Disc}(t)}\, \dfrac {d^n}{d t^n} \left[ \dfrac {\text{Disc}(t)}{n!} \right] \ \implies \ \dfrac {\text{Disc}(t + \delta)}{\text{Disc}(t)} = c_0 + c_1 \, \delta + c_2 \, \delta^2. \]
\noindent Then $\sqrt{\text{Disc}(t + \delta)} \in K$ for some nonzero $\delta \in K$ if and only if the following quantity is $K$-rational:
\[ u = \dfrac {\sqrt{\text{Disc}(t + \delta)} - \sqrt{\text{Disc}(t)}}{\delta \sqrt{\text{Disc}(t)}} \qquad \iff \qquad  \delta = \dfrac {c_1 - 2 \, u}{u^2 - c_2}  \]
\noindent because $(\delta \, u + 1)^2 = \text{Disc}(t + \delta) / \text{Disc}(t) = c_0 + c_1 \, \delta + c_2 \, \delta^2$. Inductively, if $\sqrt{\text{Disc}(t_i)} \in K$ for some $t_i = t \in K$, then we can find $\delta \in K$ as a function of $t_i$ and $u$ such that $\sqrt{\text{Disc}(t_{i+1})} \in K$ for $t_{i+1} = t + \delta \in K$. 

If $\sqrt{\text{Disc}(t \pm \delta)}$ are simultaneously $K$-rational, then define the quantities
\[ u = \dfrac {\sqrt{\text{Disc}(t + \delta)} - \sqrt{\text{Disc}(t)}}{\delta \sqrt{\text{Disc}(t)}} \ \text{and} \ v = \dfrac {\sqrt{\text{Disc}(t - \delta)} - \sqrt{\text{Disc}(t)}}{\delta \sqrt{\text{Disc}(t)}} \]
\noindent so that we have the identities
\[ \begin{aligned} \left( \dfrac {\delta \, v + 1}{\delta \, u + 1} \right)^2 & = \dfrac {\text{Disc}(t - \delta)}{\text{Disc}(t + \delta)} = \dfrac {c_0 - c_1 \, \delta + c_2 \, \delta^2}{c_0 + c_1 \, \delta + c_2 \, \delta^2} \\[5pt] & = \dfrac {u^4 + 2 \, c_1 \, u^3 + (2 \, c_2 - c_1^2) \, u^2 - 6 \, c_1 \, c_2 \, u + c_2 \, (c_2 + 2 \, c_1^2)}{(u^2 - c_1 \, u + c_2)^2}. \end{aligned} \]
\noindent This is a quartic curve, so we transform into a cubic curve using the ideas in Cassels \cite[Chapter 8]{MR1144763}.  Using the birational transformations
\[ \left. \begin{aligned}
X & = 2 \, k \, \dfrac {c_1}{v - u - c_1} \\[5pt] 
Y & = 2 \, k \, \dfrac {2 \, u - c_1}{v - u - c_1} 
\end{aligned} \right \} \qquad \iff \qquad \left \{ \begin{aligned}
u & = c_1 \, \dfrac {Y + X}{2 \, X} \\[5pt]
v & = c_1 \, \dfrac {Y + 3 \, X + 4 \, k}{2 \, X} \end{aligned} \right. \]
\noindent it is easy to see that the quartic identity above is equivalent to the cubic relation $Y^2 + 4 \, X \, Y + 4 \, k \, Y = X^3 + k \, X^2$ in terms of the $K$-rational quantity
\[ k = \dfrac {c_1^2 - 4 \, c_2}{c_1^2} = \dfrac {\text{Disc}'(t)^2 - 2 \, \text{Disc}(t) \, \text{Disc}''(t)}{\text{Disc}'(t)^2}. \]
\noindent Conversely, if $(X:Y:1)$ is a $K$-rational point on the cubic curve, then both $\sqrt{\text{Disc}(t \pm \delta)} \in K$ simultaneously when we choose the common difference
\[ \delta = \dfrac {c_1 - 2 \, u}{u^2 - c_2} = \dfrac {-c_1 - 2 \, v}{v^2 - c_2} = - \dfrac {1}{c_1} \, \dfrac {4 \, X \, Y}{Y^2 + 2 \, X \, Y + k \, X^2}. \]
\noindent This completes the proof. \end{proof}

We explain how Theorem \ref{conicsquares} is related to Proposition \ref{3squares}.  For the parabola $\mathcal C: \, y = x^2$ and the linear polynomial $\ell(x,y,1) = y$, we have the discriminant $\text{Disc}(t) = t$.  We choose a nonzero square $t_0 = x_2^2$ so that $\sqrt{\text{Disc}(t_0)} = x_2 \in K$.  As $k = 1$ in this case, the cubic curve $Y^2 + 4 \, X \, Y + 4 \, k \, Y = X^3 + k \, X^2$ is singular, so that we have a rational parametrization of $K$-rational points $(X:Y:1)$.
\[ t = -\dfrac {Y + 3 \, X + 4}{Y + X} \qquad \iff \qquad \left \{ \begin{aligned} X & = - \dfrac {1}{4} \, \dfrac {t}{(t+1)^2} \\[5pt] Y & = \dfrac {1}{16} \, \dfrac {t-1}{(t+1)^3} \end{aligned} \right. \]

\noindent This yields the common difference
\[ \delta = - \dfrac {1}{c_1} \, \dfrac {4 \, X \, Y}{Y^2 + 2 \, X \, Y + k \, X^2} = \dfrac {4 \, (t^3 - t)}{(t^2 + 1)^2} \cdot x_2^2 \]

\noindent for the $K$-rational squares
\[ \begin{matrix} 
y_1 & = & t_0 - \delta & = & \left( \dfrac {t^2 - 2 \, t - 1}{t^2 + 1} \cdot x_2 \right)^2, \\[10pt]
y_2 & = & t_0 & = & x_2^2, \\
y_3 & = & t_0 + \delta & = & \left( \dfrac {t^2 + 2 \, t - 1}{t^2 + 1} \cdot x_2 \right)^2.  \\
\end{matrix} \]

\subsection{Relation with Modular Curves and Moduli Spaces} 

There is quite a bit of symmetry in the general construction above.

\begin{symmetry} \label{symmetry} Continue notation as in Proposition \ref{conic} and Theorem \ref{conicsquares}, and assume that
\[ \mathcal E_k: \quad Y^2 + 4 \, X \, Y + 4 \, k \, Y = X^3 + k \, X^2 \]
\noindent is an elliptic curve. Translation $(X:Y:1) \mapsto (0:0:1) \oplus (X:Y:1)$ yields a permutation $\sigma$ of order 4, while inversion $(X:Y:1) \mapsto [-1] \, (X:Y:1)$ yields a permutation $\tau$ of order 2.  The group
\[ D_4 \simeq \left \langle \sigma, \, \tau \, \bigl| \, \sigma^4 = \tau^2 = 1, \ \tau \circ \sigma \circ \tau = \sigma^{-1} \right \rangle \]
\noindent acts on $\mathcal E_k(K)$, and sends $\delta \mapsto \pm \delta$. \end{symmetry}

\begin{proof} It is easy to verify that $(0:0:1) \in \mathcal E_k(K)$ is a torsion point of order 4.  It is also straight-forward to verify the following maps:
\[ \sigma: \left \{ \begin{aligned} X & \mapsto - 4 \, k \, Y/X^2 \\[4pt] Y & \mapsto 4 \, k^2 \, (X^2 - 4 \, Y)/X^3 \\[4pt] u & \mapsto (c_1 \, v + 2 \, c_2)/(2 \, v + c_1) \\[5pt] v & \mapsto -u \\[5pt] \delta & \mapsto -\delta \end{aligned} \right \}
\quad
\tau: \left \{ \begin{aligned} X & \mapsto X \\[5pt] Y & \mapsto -Y - 4 \, X - 4 \, k \\[5pt] u & \mapsto -v \\[5pt] v & \mapsto -u \\[5pt] \delta & \mapsto - \delta \end{aligned} \right \} \]

\noindent Hence $\sigma^4 = \tau^2 = 1$ while $\tau \circ \sigma \circ \tau = \sigma^{-1}$. \end{proof}

Proposition \ref{symmetry} hints at the many properties of $\mathcal E_k$ from Theorem \ref{conicsquares}.  This cubic curve is quite general.

\begin{4torsion} \label{4torsion} Let $K$ be a field of characteristic different from 2.  Then any elliptic curve $\mathcal E$ over $K$ with a $K$-rational point $P$ of order 4 must be in the form
\[ \mathcal E_k: \quad Y^2 + 4 \, X \, Y + 4 \, k \, Y = X^3 + k \, X^2 \]
\noindent where $k \neq 0, \, 1, \, \infty$ and $P_k = (0:0:1)$. \end{4torsion}

\begin{proof} This can be found in Kubert \cite[Page 217, Table 3]{MR0434947} and Husem{\"o}ller \cite[Page 94, Example (4.6)(a)]{Husemoller:2004p9376}, but we provide a relatively self-contained proof.  Without loss of generality, assume that $P = (x_0 : y_0 : 1)$ is a point of order 4 on the Weierstrass equation
\[ \mathcal E: \quad y^2 + a_1 \, x \, y + a_3 \, y = x^3 + a_2 \, x_2 + a_4 \, x + a_6. \]
\noindent If we define the coefficients
\[ \begin{aligned}
b_2 & = a_1^2 + 4 \, a_2 \\[5pt]
b_4 & = a_1 \, a_3 + 2 \, a_4 \\[5pt]
b_6 & = a_3^2 + 4 \, a_6 \\[5pt]
b_8 & = a_1^2 \, a_6  + 4 \, a_2 \, a_6 - a_1 \, a_3 \, a_4 + a_2 \, a_3^2 - a_4^2
\end{aligned} \]
\noindent and then substitute
\[ \begin{aligned}
k_1 & = \dfrac {2 \, y_0 + a_1 \, x_0 + a_3}{6 \, x_0^2 + b_2 \, x_0 + b_4} \\[5pt]
k_2 & = 16 \, \dfrac {3 \, x_0^4 + b_2 \, x_0^3 + 3 \, b_4 \, x_0^2 + 3 \, b_6 \, x_0 + b_8}{( 6 \, x_0^2 + b_2 \, x_0 + b_4)^2} \\[5pt]
k_3 & = 16 \, \dfrac {( 2 \, y_0 + a_1 \, x_0 + a_3)^4}{( 6 \, x_0^2 + b_2 \, x_0 + b_4)^3} \\[5pt]
X & = 16 \, k_1^2 \, (x-x_0) \\[5pt] 
Y & = 64 \, k_1^3 \, (y-y_0) + 32 \, k_1^2 \, (a_1 \, k_1 - 1) \, (x - x_0)
\end{aligned} \]
\noindent then $\mathcal E$ has the equation $Y^2 + 4 \, X \, Y + 4 \, k_3 \, Y = X^3 + k_2 \, X^2$ while $P$ maps to the point $(0:0:1)$. It is easy to verify that
\[ \begin{aligned} \, [4] \, \bigl( 0:0:1 \bigr) = \biggl( & 4 \, k_2 \, (k_3 - k_2) \, (16 \, k_3^2 - 16 \, k_3 \, k_2 + k_2^3) \\ & : k_2^3 \, (32 \, k_3^2- 48 \, k_3 \, k_2 + 16 \, k_2^2 + k_2^3 \bigr) : 64 \, (k_3 - k_2)^2 \biggr), \end{aligned} \]

\noindent and so $(0:0:1)$ has order 4 if and only if $k_2 = k_3$.  Upon denoting $k = k_2 = k_3$, we see that $\mathcal E \simeq \mathcal E_k$ is an elliptic curve if and only if $k \neq 0, \, 1, \, \infty$. \end{proof}

We offer a general discussion regarding Proposition \ref{4torsion} in order to explain the general framework.  We introduce modular curves and moduli spaces.  To this end, recall that the extended upper-half plane $\mathbb H = \bigl \{ x + i \, y \ \bigl| \ y > 0 \bigr \} \cup \mathbb P^1(\mathbb Q)$ is isomorphic to the unit disk via the map $\mathbb H \to \mathbb P^1(\mathbb C)$ which sends $\tau$ to $q = e^{2 \pi i \tau}$, and we have an left action
\[ \circ: SL_2(\mathbb Z) \times \mathbb H \to \mathbb H \quad \text{defined by} \quad \left[ \begin{matrix} a & b \\ c & d \end{matrix} \right] \circ \tau = \dfrac {a \, \tau + b}{c \, \tau + d}. \]
\noindent For any positive integer $N$, we consider the congruence subgroups
\[ \begin{aligned} 
\Gamma_0(N) & = \left \{ \left. \left[ \begin{matrix} a & b \\ c & d \end{matrix} \right] \in SL_2(\mathbb Z) \ \right| \ c \equiv 0 \pmod{N} \right \} \\[5pt]
\Gamma_1(N) & = \left \{ \left. \left[ \begin{matrix} a & b \\ c & d \end{matrix} \right] \in SL_2(\mathbb Z) \ \right| \ a \equiv d \equiv 1, \ c \equiv 0 \pmod{N} \right \}.
\end{aligned} \]

\noindent We have a commutative diagram
\[ \xymatrix @!R=0.2in @!C=0.7in{
\dfrac {\mathbb H}{\Gamma_1(4)} \ar@{->}[r]^{k} \ar@{->}[d] & \mathbb P^1(\mathbb C) \ar@{->}[d] & & k(\tau) = - \dfrac {1}{\displaystyle 16 \, q \prod_{n=1}^{\infty} \bigl( 1 + q^n \bigr)^8 \, \bigl( 1+q^{2n} \bigr)^8} \\
\dfrac {\mathbb H}{\Gamma_1(2)} \ar@{->}[r]^{r} \ar@{->}[d] & \mathbb P^1(\mathbb C) \ar@{->}[d] & & r(\tau) = \dfrac {1}{\displaystyle q \prod_{n=1}^{\infty} \bigl ( 1 + q^n \bigr)^{24}} \\
\dfrac {\mathbb H}{\Gamma_1(1)} \ar@{->}[r]^{j} & \mathbb P^1(\mathbb C) & & j(\tau) = \dfrac {\displaystyle \left[ 1 + 240 \sum_{n=1}^{\infty} \sigma_3(n) \, q^n \right]^3}{\displaystyle q \prod_{n=1}^{\infty} \bigl( 1 - q^n \bigr)^{24}} \\
} \]

\noindent where $\sigma_\alpha(n) = \sum_{d|n} d^\alpha$ is the divisor function, the horizontal arrows are bijections, and we have the identities $j(\tau) = \bigl( r(\tau) + 256 \bigr)^3 / r(\tau)^2$ and $r(\tau) = 16 \, k(\tau)^2 / \bigl( 1 - k(\tau) \bigr)$ coming from the vertical arrows.  These maps are defined over $\mathbb C$, but we wish to give an interpretation over any field $K$ of characteristic different from 2.  The moduli spaces $X_0(N)$ and $X_1(N)$ consist of pairs $(\mathcal E, C)$ and $(\mathcal E, P)$ of elliptic curves $\mathcal E$ with either a cyclic subgroup $C$ of order $N$ or a specified point $P$ of order $N$, respectively.  We have $X_0(N)(\mathbb C) \simeq \mathbb H / \Gamma_0(N)$ and $X_1(N)(\mathbb C) \simeq \mathbb H / \Gamma_1(N)$, but we are interested in $K$-rational points.  By ``forgetting'' the level structures, we have a commutative diagram
\[ \xymatrix @!R=0.0in @!C=0.7in{
\mathbb P^1(K) \ar@{->}[r] \ar@{->}[d] & X_1(4)(K) \ar@{->}[d] & & {\begin{aligned} \mathcal E_k &: \ Y^2 +4 \, X \, Y + 4 \, k \, Y = X^3 + k \, X^2 \\ P_k & = (0:0:1) \end{aligned}} \\
\mathbb P^1(K) \ar@{->}[r] \ar@{->}[d] & X_1(2)(K) \ar@{->}[d] & & {\begin{aligned} \mathcal E_r &: \ Y^2 = X^3 + 2 \, X^2 + r/(r+64) \, X \\ Q_r & = (0:0:1) \end{aligned}} \\
\mathbb P^1(K) \ar@{->}[r] & X_1(1)(K) & & {\begin{aligned} \mathcal E_j &: \ Y^2 + X \, Y = X^3 \\ & \quad + 36/(1728 - j) \, X + 1/(1728 - j) \\ \mathcal O_j & = (0:1:0) \end{aligned}} \\
} \]

\noindent where the horizontal arrows sending $k \mapsto \bigl( \mathcal E_k, \, P_k \bigr)$ and $j \mapsto \bigl( \mathcal E_j, \, \mathcal O_j \bigr)$ are bijections as in Proposition \ref{4torsion}, and we have the two relations $j = (r+256)^3/r^2$ and $r = 16 \, k^2 / (1-k)$ coming from the vertical arrows.  (Not every $X_0(N)$ or $X_1(N)$ is birationally equivalent to $\mathbb P^1$: as mentioned in the previous section, $X_0(24)$ is the elliptic curve $Y^2 = X^3 + 5 \, X^2 + 4 \, X$.)

\subsection{Congruent Numbers Revisited}

Theorem \ref{conicsquares} gives an explicit way to write down arithmetic progressions $\{ P_1, \, P_2, \, P_3 \} \subseteq \mathbb P^2(K)$ on a conic section
\[ \mathcal C = \left \{ (x_1: x_2 : x_0) \in \mathbb P^2 \ \left| \ \begin{aligned} A \, x_1^2 & + 2 \, B \, x_1 \, x_2 + C \, x_2^2 \\ & + 2 \, D \, x_1 \, x_0 + 2 \, E \, x_2 \, x_0 + F \, x_0^2 = 0 \end{aligned} \right. \right \} \]

\noindent with respect to a linear rational map 
\[ \begin{CD} \mathcal C(K) @>{\ell}>> \mathbb P^1(K), \qquad (x_1 : x_2 : x_0) \mapsto \dfrac {a \, x_1 + b \, x_2 + c \, x_0}{d \, x_1 + e \, x_2 + f \, x_0} \end{CD} \]

\noindent defined over a field $K$ of characteristic different from 2.  Consider those $t_0 \in K$ such that $\sqrt{\text{Disc}(t_0)} \in K$ and
\[ k = \dfrac {\text{Disc}'(t_0)^2 - 2 \, \text{Disc}(t_0) \, \text{Disc}''(t_0)}{\text{Disc}'(t_0)^2} \neq 0, \, 1, \, \infty \]

\noindent in terms of the discriminant
\[ \begin{aligned} & \text{Disc}(t) \\ & \quad = \left[ \begin{matrix} a-d \, t \\ b-e \, t \\ c-f \, t \end{matrix} \right]^T \left[ \begin{matrix} E^2 - C \, F & B \, F - D \, E & C \, D - B \, E \\ B \, F - D \, E & D^2 - A \, F & A \, E - B \, D \\ C \, D - B \, E & A \, E - B \, D & B^2 - A \, C \end{matrix} \right] \left[ \begin{matrix} a-d \, t \\ b-e \, t \\ c-f \, t \end{matrix} \right]. \end{aligned} \]

\noindent For each $K$-rational point $(X:Y:1)$ on the elliptic curve
\[ \mathcal E_k: \quad Y^2 + 4 \, X \, Y + 4 \, k \, Y = X^3 + k \, X^2 \]

\noindent the desired set is
\[ \{ P_1, \, P_2, \, P_3 \} = \left \{ \begin{matrix} \bigl( x_1(t_0 - \delta):x_2(t_0-\delta):x_0(t_0-\delta) \bigr), \\[10pt] \bigl( x_1(t_0):x_2(t_0):x_0(t_0) \bigr), \\[10pt] \bigl( x_1(t_0 + \delta):x_2(t_0+\delta):x_0(t_0+\delta) \bigr) \end{matrix} \right \} \]

\noindent in terms of the coordinates

\[ \begin{aligned}
x_1(t) & = B \, (b-e \, t) \, (c-f \, t) - C\, (a-d \, t) \, (c-f \, t) \\ & - D \, (b-e \, t)^2 + E \, (a-d \, t) \, (b-e \, t) \pm (b - e \, t) \, \sqrt{\text{Disc}(t)}, \\[5pt]
x_2(t) & = -A \, (b - e \, t) \, (c - f \, t ) + B \, (a-d \, t ) \, (c-f \, t ) \\ & + D \, (a-d \, t ) \, (b-e \, t) - E \, (a-d \, t)^2 \mp (a-d \, t)  \, \sqrt{\text{Disc}(t)}, \\[5pt]
x_0(t) & = A \, (b-e \, t )^2 - 2 \, B \, (a-d \, t) \, (b-e \, t) + C \, (a-d \, t )^2;
\end{aligned} \]

\noindent and the common difference
\[ \delta = - \dfrac {\text{Disc}(t_0)}{\text{Disc}'(t_0)} \, \dfrac {4 \, X \, Y}{Y^2 + 2 \, X \, Y + k \, X^2}. \]

\noindent We have seen that the dihedral group $D_4$ of order 8 coming from translation by 4-torsion on $\mathcal E_k$ acts on the points $(X:Y:1) \in \mathcal E_k(K)$ and sends $\delta \mapsto \pm \delta$.  We generalize Corollary \ref{congruum} by asking which $\delta \in K$ can occur as the common differences of arithmetic progressions on conic sections with respect to linear rational maps.

\begin{coniccongruum} \label{coniccongruum} For each nonzero $\delta \in K$, denote the quantities
\[ \begin{matrix}
A & = & \text{Disc}(t) - \text{Disc}(t-\delta) & = & \text{Disc}'(t) \, \delta - \text{Disc}''(t) \, \delta^2 / 2, \\[5pt]
B & = & \text{Disc}(t+\delta) - \text{Disc}(t) & = & \text{Disc}'(t) \, \delta + \text{Disc}''(t) \, \delta^2 / 2, \\[5pt]
C & = & \text{Disc}(t+\delta) - \text{Disc}(t-\delta) & = & 2 \, \text{Disc}'(t) \, \delta.
\end{matrix} \]
\noindent Then the following are equivalent.
\begin{enumerate}
\item There exist $x_1, \, x_2, \, x_3 \in K$ such that $x_2^2 - A = x_1^2$ and $x_2^2 + B = x_3^2$.
\item There exist $X, \, Y \in K$ with $Y \neq 0$ such that $Y^2 = X \, (X-A) \, (X+B)$.
\item Upon defining an ``angle'' $\theta$ by $\cos \theta = (B-A)/C$, there exist $a, \, b, \, c \in K$ such that $a^2 - 2 \, a \, b \, \cos \theta + b^2 = c^2$ and $(1/2) \, a \, b \, \sin \theta = \sqrt{A \, B}$.
\end{enumerate} \end{coniccongruum}

\begin{proof} Following the proof of Theorem \ref{conicsquares}, the quantities $A$ and $B$ are chosen so that we have the identities
\[ \begin{aligned}
x_1 & = \sqrt{\text{Disc}(t- \delta)} \\[5pt]
x_2 & = \sqrt{\text{Disc}(t)} \\[5pt]
x_3 & = \sqrt{\text{Disc}(t+\delta)}
\end{aligned} \qquad \implies \qquad \begin{aligned}
x_1^2 & = x_2^2 - A, \\[5pt]
x_3^2 & = x_2^2 + B.
\end{aligned} \]

\noindent This motivates the following transformations:
\[ \left. \begin{aligned}
X & = \dfrac {A \, B \, (x_1 - x_3)}{B \, x_1 - C \, x_2 + A \, x_3} \\[5pt]
Y & = \dfrac {A \, B \, C}{B \, x_1 - C \, x_2 + A \, x_3}
\end{aligned} \right \} \iff \left \{ \begin{aligned}
x_1 & = \dfrac {X^2 - 2 \, A \, X - A \, B}{2 \, Y} \\[5pt]
x_2 & = \dfrac {X^2 + A \, B}{2 \, Y} \\[5pt]
x_3 & = \dfrac {X^2 + 2 \, B \, X - A \, B}{2 \, Y}
\end{aligned} \right. \]

\noindent Hence $A = x_2^2 - x_1^2$ and $B = x_3^2 - x_2^2$ if and only if $Y^2 = X \, (X-A) \, (X+B)$ for some nonzero $Y$. Similarly, we have the following transformation:
\[ \left. \begin{aligned} a & = \dfrac {X^2 +(B-A) \, X - A \, B}{Y} \\[5pt] b & = \dfrac {C \, X}{Y} \\[5pt] c & = \dfrac {X^2 + A \, B}{Y} \end{aligned} \right \} \iff \left \{ \begin{aligned} X & = \dfrac {2 \, A \, B \, b}{(B-A) \, b + C \, (c-a)} \\[5pt] Y & = \dfrac {2 \, A \, B \, C}{(B-A) \, b + C \, (c-a)} \end{aligned} \right. \]

\noindent Hence $Y^2 = X \, (X-A) \, (X+B)$ for some nonzero $Y$ if and only if $a^2 - 2 \, a \, b \, \cos \, \theta + b^2 = c^2$ and $a \, b \, \sin \theta = 2 \, \sqrt{A \, B}$. \end{proof}

We explain how Corollary \ref{coniccongruum} is related to Corollary \ref{congruum}.  For the parabola $\mathcal C: \, y = x^2$ and the linear polynomial $\ell(x,y,1) = y$, we have the discriminant $\text{Disc}(t) = t$.  Then $A = B = \delta$ and $C = 2 \, \delta$ in this case, so that we have the elliptic curve $Y^2 = X^3 - \delta^2 \, X$.

In general, common differences $\delta$ for a 3-term arithmetic progression $\{ P_1, \, P_2, \, P_3 \}$ on a conic section $\mathcal C$ with respect to a linear rational map $\ell$ correspond to nontrivial rational points $(X:Y:1)$ on the Frey curve $Y^2 = X \, (X-A) \, (X+B)$.  Geometrically, we can construct a $K$-rational triangle having rational sides of length $a$, $b$, and $c$ and having a common angle $\theta$ such that the area is $\sqrt{A \, B}$.  This is very similar to the concept of a $\theta$-congruent number as defined by Fujiwara \cite{MR99c:11068}.  We have decided not to use this language in the exposition at hand because the area $\sqrt{A \, B}$ is not always a linear function of the difference $\delta$.


\end{document}